\documentclass[a4paper, 10pt]{amsart}
\usepackage{amsmath,amssymb,amsfonts,amsthm,latexsym,mathrsfs}
\usepackage[utf8]{inputenc}
\usepackage{lmodern}
\usepackage[T1]{fontenc}
\usepackage{enumerate}
\theoremstyle{plain}
\newtheorem{theorem}[subsection]{{\bf Theorem}}
\newtheorem*{theorem*}{{\bf Theorem}}
\newtheorem{corollary}[subsection]{{\bf Corollary}}
\newtheorem*{corollary*}{{\bf Corollary}}
\newtheorem{proposition}[subsection]{{\bf Proposition}}
\newtheorem{lemma}[subsection]{{\bf Lemma}}

\theoremstyle{definition}

\theoremstyle{remark}

\numberwithin{equation}{subsection}

\DeclareMathOperator\exprank{exprank}
\begin{document}
\baselineskip=15pt
\title{On $p$-central groups}
\author{Urban Jezernik}
\address{Department of Mathematics \\ University of Ljubljana \\ Jadranska 21 \\ 1000 Ljubljana \\ Slovenia}
\email{urban.jezernik@gmail.com}
\date{June 9, 2012}
\maketitle
\begin{abstract}
 We extend the notion of free $p$-central groups for odd primes $p$ to the case $p=2$ by introducing a variant of the lower $p$-central series. This enables us to calculate Schur multipliers of free $p$-central groups. We also prove that for any $p$-central group the exponent of its Schur multiplier divides the exponent of the group, and determine its exponential rank. 
\end{abstract}
\section{Introduction}
\noindent Let $G$ be a finite $p$-group and denote by $\Omega_i(G)$ the subgroup of $G$ generated by the elements of order dividing $p^i$. We call $G$ a {\em $p$-central} group if $p$ is odd and $\Omega_1(G) \leq Z(G)$, or if $p=2$ and $\Omega_2(G) \leq Z(G)$. These groups are the dual counterpart of powerful $p$-groups, which have proved immensely useful in the study of finite $p$-groups and pro-$p$ groups \cite{lubotzkyMannPowerful, dixonSautoyAnalytic}. 

The power structure of $p$-central groups has been explicitly studied by several authors, beginning with Laffey \cite{laffey1, laffey2}, and more recently by Mann \cite{mannPower}. It has been shown that these groups possess certain properties of regular $p$-groups, for example, the subgroup $\Omega_i(G)$ is equal to the subset of all elements of order diving $p^i$. However, unlike powerful $p$-groups, $p$-central groups need not be regular (cf. \cite{bubboloni0}). A more general approach to the study of these groups has been undertaken by Bubboloni and Corsi Tani in \cite{bubboloni}, where free $p$-central groups are constructed for odd primes $p$. Their work is based on the properties of the well-known lower $p$-central series $\lambda_i(F_r)$ of the free group $F_r$ on $r$-generators \cite[VIII]{huppertBlackburn}. More specifically, they proved that the group $^n G_r = F_r/\lambda_{n+1}(F_r)$ is $p$-central of exponent $p^n$, and that any $r$-generated $p$-group of exponent $p^n$ is its homomorphic image.

In this note, we extend the results of \cite{bubboloni} to the even prime by appropriately adapting the lower $p$-central series $\lambda_i$ to a variant $\bar \lambda_i$, defined inductively by $\bar\lambda_1(G) = G$ and $\bar \lambda_{n+1}(G) = \bar\lambda_n(G)^4 [\bar\lambda_n(G), G]$. This is of course motivated by the definition of $p$-central groups. After establishing some basic properties of this series, we put $^n G_r = F_r/\bar \lambda_{n+1}(F_r)$ and prove the following theorem.

\begin{theorem}
The group ${}^n G_r$ is an $r$-generated $2$-central group of exponent $4^n$. Moreover, any $r$-generated $2$-central group of exponent $2^n$ is a homomorphic image of ${}^{\lceil \frac{n}{2} \rceil + 1} G_r$, and more generally: any finite $r$-generated $2$-group is a homomorphic image of ${}^m G_r$ for some $m$.
\end{theorem}

We obtain several properties of the group $^n G_r$ which are known to hold in the odd case. This is done by studying the $\bar \lambda$-series in more detail, the results of which also enable us to provide an explicit formula for the Schur multiplier of free $p$-central groups via Hopf's formula.

\begin{theorem}
 The Schur Multiplier $M({}^n G_r)$ is elementary abelian (resp. free $\mathbb Z_4$-module for $p=2$) of rank $\sum_{i=2}^{n+1} c(r,i)$, where  $c(w,g)$ is the number of basic commutators of weight $w$ on $g$ generators in a fixed sequence.
\end{theorem}

For general $p$-central groups, we prove a classical result concerning the exponent of their Schur multipliers. The theorem is an affirmation of a special case of a conjecture due to Schur, stating that the exponent of $M(G)$ divides $\exp G$ for every finite group $G$. By homological arguments, it suffices to consider this question for finite $p$-groups. The conjecture has turned out to be false in general (cf. \cite{pmSchurMultPowerEnd}), yet it is still not known whether or not it holds true for odd primes $p$. We prove that it is so for the class of $p$-central groups.

\begin{theorem}
 Let $G$ be a $p$-central group. Then $\exp M(G) \leq \exp G$.
\end{theorem}

Our final contribution to the study of $p$-central groups is determining their exponential rank, recently introduced by Moravec in \cite{pmSchurMultPowerEnd}. This provides yet another link between $p$-central groups and powerful $p$-groups, as the results are completely the same, and indicates a potentially profounder role of exponential rank. The theorem goes as follows.

\begin{theorem}
 Let $G$ be a $p$-central group.
\begin{enumerate}[\indent \em (a)]
 \item If $p$ is odd, then $\exprank(G) = 0$.
 \item If $p=2$ and $G$ is not abelian, then $\exprank(G) = 1$.
\end{enumerate}
\end{theorem}

\section{The $\bar \lambda$-series}
\label{lambdaSeries}

\noindent Let the {\em $\bar \lambda$-series} of a group $G$ be defined recursively by 
\[
\bar\lambda_1(G) = G, \qquad \bar \lambda_{n+1}(G) = \bar\lambda_n(G)^4 [\bar\lambda_n(G), G].
\]
This is a descending series of characteristic subgroups of $G$. Following \cite{huppertBlackburn}, we also denote 
\[
N_{n,k}(G) = \gamma_k(G)^{4^{n-k}} \gamma_{k-1}(G)^{4^{n-k-1}} \cdots \gamma_n(G)
\]
for $1 \leq k \leq n$. Our first aim is to show that the groups $\bar \lambda_i$ and $N_{i,j}$ are tightly correlated. We begin with a lemma.


\begin{lemma} \label{lemma-Nseries}
 Let $G$ be a group and $m,n,k$ positive integers. 
\begin{enumerate}[\indent \em (a)]
  \item For any $1 \leq i \leq n$, we have $[\gamma_i(G)^{4^{n-i}}, G] \leq N_{n+1,i+1}(G)$.

  \item  $[N_{n,k}(G), G] = N_{n+1,k+1}(G)$ and $N_{n,k}(G)^{4^{m}} \leq N_{n+m,k}(G)$. 
\end{enumerate}
\end{lemma}
\begin{proof}
 Using \cite[Lemma 1.2]{huppertBlackburn}, we have
 \[
 [\gamma_i(G)^{4^{n-i}}, G] \leq \prod_{r=0}^{2(n-i)} \gamma_{1+i2^r}(G)^{2^{2(n-i)-r}}.
 \]
Taking two consecutive terms of this product into account, the inclusion
\[
 \gamma_{1 + i2^{2s+1}}(G)^{2^{2(n-i) - (2s+1)}}  \gamma_{1 + i2^{2s+2}}(G)^{2^{2(n-i) - (2s+2)}} \leq  \gamma_{i + s + 2}(G)^{4^{(n-i) - (s+1)}}
\]
holds for any $0 \leq s \leq n-i-1$, since $1 + i2^{2s+1} \geq i + s + 2$. This proves part (a) of the lemma, from which the inclusion $[N_{n,k}(G), G] \leq N_{n+1, k+1}(G)$ of the first part of (b) follows directly. For the other inclusion, we refer the reader to the proof of Theorem 2.4 of \cite{blackburnEvens}, the argument is essentially the same. The second part of (b) follows from the classical Hall-Petrescu formula: $ N_{n,k}(G)^4 \leq N_{n+1,k}(G) \cdot [N_{n,k}(G), N_{n,k}(G)] \leq N_{n+1, k}(G)$.
\end{proof}

An explicit formula for the $\bar\lambda$-series is thus at hand.

\begin{proposition} \label{prop-lambdaDescr}
 Let $G$ be a group. For any $n$, we have $\bar \lambda(G)_n = N_{n,1}(G)$.
\end{proposition}
\begin{proof}
 It suffices to verify that the series $N_{n,1}(G)$ satisfies the recursion formula for the $\bar \lambda$-series, and we do this by induction.  We plainly have $\gamma_i(G)^{4^{n-i+1}} \leq (\gamma_i(G)^{4^{n-i}})^{4} \leq N_{n,1}(G)^4$ for any $1 \leq i \leq n$ and $\gamma_{n+1}(G) = [\gamma_n(G), G] \leq [N_{n,1}(G), G]$, which implies $N_{n+1, 1}(G) \leq N_{n,1}(G)^4 [N_{n,1}(G), G]$. The reverse inclusion follows from Lemma \ref{lemma-Nseries}.
\end{proof}

When the nilpotent quotients of the given group $G$ are torsion-free (for example when $G$ is free), another connection between the groups $\bar \lambda_i(G)$ and $N_{i,j}(G)$ exists. In the proof, we will use a well-known consequence of the Hall-Petrescu formula, which we state beforehand.

\begin{lemma} \label{lemma-hallPet}
Let $G$ be a group and $x,y \in G$. Suppose $p$ is a prime and $k$ a positive integer. Then
\[
(xy)^{p^k} \equiv x^{p^k} y^{p^k} \mod{\gamma_2(\langle x,y \rangle)^{p^k} \prod_{i=1}^k \gamma_{p^i}(\langle x,y \rangle)^{p^{k-i}}}.
\]
If furthermore $x$ and $[x,y]$ belong to $H \leq G$, then
\[
 [x^{p^k},y] \equiv [x,y]^{p^k} \mod{\gamma_2(H)^{p^k} \prod_{i=1}^k \gamma_{p^i}(H)^{p^{k-i}}}.
\]
\end{lemma}

\begin{proposition} \label{prop-torsionFree1}
 Let $G$ be a group and $n,k$ positive integers.
\begin{enumerate}[\indent \em (a)]
   \item For any $x,y \in \gamma_{k-1}(G)$, we have  
\[ 
(xy)^{4^{n-k+1}} \equiv x^{4^{n-k+1}} y^{4^{n-k+1}} \mod{N_{n,k}(G) \cap N_{n+1, k-1}(G)}.
\]

   \item Any element of $N_{n,k}(G)$ can be written in the form $a_k^{4^{n-k}} a_{k+1}^{4^{n-k-1}} \cdots a_n$ for suitable $a_j \in \gamma_j(G)$.

   \item Suppose that the groups $G/\gamma_j(G)$ are torsion-free for all $1 \leq j \leq n$. Then $\bar\lambda_{n}(G) \cap \gamma_k(G) = N_{n,k}(G)$.
\end{enumerate}
\end{proposition}
\begin{proof}
 Note first that $\gamma_a(G)^{4^b} \leq N_{n,k}(G)$ whenever $a \geq k$ and $a+b \geq n$.

(a) We use Lemma \ref{lemma-hallPet}. Note that $\gamma_2(\langle x,y \rangle) \leq \gamma_{2(k-1)}(G)$. Taking two consecutive terms of the product $\prod_{i=1}^{2(n-k+1)} \gamma_{2^i(k-1)}(G)^{2^{2(n-k+1)-i}}$ into account, we see that for any $0 \leq s \leq n-k$, the group
\[
 \gamma_{2^{2s+1}(k-1)}(G)^{2^{2(n-k+1)-(2s+1)}} \gamma_{2^{2s+2}(k-1)}(G)^{2^{2(n-k+1)-(2s+2)}}
\]
is contained in $\gamma_{2^{2s+1}(k-1)}(G)^{2^{2(n-k+1)-(2s+2)}}$, which is itself a subgroup of the intersection $N_{n,k}(G) \cap N_{n+1, k-1}(G)$, since $2^{2s+1}(k-1) + (n-k+1)-(s+1) \geq 2^{2s+1} + n - s \geq n+1$.

 (b) This is proved by reverse induction on $k$, the case $k=n$ being trivial. Suppose $x \in N_{n,k-1}(G) = \gamma_{k-1}(G)^{4^{n-k+1}} \gamma_{k}(G)^{4^{n-k}} \cdots \gamma_n(G)$. Then there exist elements $x_1, \dots, x_l \in \gamma_{k-1}(G)$ such that $x = x_1^{4^{n-k-1}} x_2^{4^{n-k-1}} \cdots x_l^{4^{n-k-1}} c$, where $c \in N_{n,k}(G)$. We finish off using (a) followed by the induction hypothesis.

(c) We prove the claim by induction on $k$. First note that the inclusion $N_{n,k}(G) \leq \bar\lambda_{n}(G) \cap \gamma_k(G)$ holds without the additional assumptions. Now take any $x \in \bar\lambda_{n}(G) \cap \gamma_k(G)$. By the induction hypothesis and (a), there exist elements $a \in \gamma_{k-1}(G)$ and $y \in N_{n, k}(G)$ such that $x = a^{4^{n-k+1}} y$. This implies  $a^{4^{n-k+1}} \in \gamma_k(G)$, and as the group $\gamma_{k-1}(G)/\gamma_k(G)$ is torsion-free, we have $a \in \gamma_k(G)$ and thus $x \in N_{n+1,k}(G)$.
\end{proof}

We now prove further properties of the $\bar \lambda$-series which are analogues of the well-known properties of the lower $p$-central series.

\begin{proposition}
 Let $G$ be a group and $m,n$ positive integers. Then
 \[
 \bar\lambda_n(G)^{4^m} \leq \bar\lambda_{n+m}(G) \quad \text{and} \quad [\bar\lambda_n(G), \bar\lambda_m(G)] \leq \bar\lambda_{n+m}(G).
 \]
\end{proposition}
\begin{proof}
The first inclusion is clear. We prove the second one by induction on $m$. By the Three Subgroups Lemma, $[\bar\lambda_n(G), [\bar\lambda_m(G), G]] \leq \bar\lambda_{m+n+1}(G)$. It remains to prove that $[\bar\lambda_n(G), \bar\lambda_m(G)^4]$ is contained in $\bar \lambda_{n+m+1}(G)$. Picking any $x \in \bar\lambda_m(G)$ and $y \in \bar\lambda_n(G)$, we have $[x,y] \in \bar\lambda_{n+m}(G)$ and $[x,y,x] \in \bar\lambda_{n+2m}(G)$ by induction. Putting $H = \langle x, [x,y] \rangle$ and using Lemma \ref{lemma-hallPet}, we get
\[
 [x^4, y] \equiv [x,y]^4 \mod{ \gamma_2(H)^2 \gamma_4(H) }.
\]
Now $[x,y]^4 \in \bar\lambda_{n+m}(G)^4 \leq \bar\lambda_{n+m+1}(G)$ and $\gamma_2(H) \leq \bar\lambda_{n+2m}(G)$, which implies $\gamma_2(H)^2 \gamma_4(H) \leq \gamma_2(H) \leq \bar\lambda_{n+m+1}(G)$. This concludes our proof.
\end{proof}

Continuing in the manner of \cite{bubboloni}, we provide a version of \cite[VIII, Theorem 1.8]{huppertBlackburn}.

\begin{lemma} \label{lemma-independent}
 Let $G$ be a group and $n$ a positive integer. If $x \in \gamma_j(G)$ and $y \in \gamma_j(G)^4 \gamma_{j+1}(G)$ for some $1 \leq j \leq n$, then $(xy)^{4^{n-j}} \equiv x^{4^{n-j}} \mod{ N_{n+1, j}(G) }$.
\end{lemma}
\begin{proof}
 Both $x$ and $y$ are elements of $\gamma_i(G)$. By Proposition \ref{prop-torsionFree1} with $k=i+1$, we have $(xy)^{4^{n-i}} \equiv x^{4^{n-i}} y^{4^{n-i}} \mod{N_{n+1,i}(G)}$. Since $y \in N_{i+1,i}(G)$, we also have $y^{4^{n-i}} \in N_{n+1,i}(G)$ by Proposition \ref{lemma-Nseries}, hence the lemma.
\end{proof}

\begin{theorem} \label{thm-lambdaSeries}
  Let $G$ be a group and $n,k$ positive integers.
  \begin{enumerate}[\indent \em (a)]
   \item There is an epimorphism $\beta$ from the direct product
   \[
   \frac{\gamma_k(G)}{\gamma_k(G)^4 \gamma_{k+1}(G)} \times \frac{\gamma_{k+1}(G)}{\gamma_{k+1}(G)^4 \gamma_{k+2}(G)} \times \cdots \times \frac{\gamma_{n}(G)}{\gamma_{n}(G)^4\gamma_{n+1}(G)}
   \]
   onto the quotient $N_{n,k}(G)/N_{n+1,k}(G)$, given by 
\[ 
\beta(\bar{a}_k, \bar{a}_{k+1}, \dots, \bar{a}_{n}) = a_k^{4^{n-k}} a_{k+1}^{4^{n-k-1}} \cdots a_{n} N_{n+1,k}(G),
\]
 where $a_j \in \gamma_j(G)$ and $\bar{a}_j = a_j \gamma_j(G)^4 \gamma_{j+1}(G)$.
   
   \item Suppose that the groups $G/\gamma_j(G)$ are torsion-free for all $1 \leq j \leq n$. Then the map $\beta$ of {\em (a)} is an isomorphism.
  \end{enumerate}
\end{theorem}
\begin{proof}
(a) It follows from the previous Lemma that the mapping $\beta$ is well-defined. It is a homomorphism by Proposition \ref{prop-torsionFree1} (a), and it is surjective by (b) of the same proposition.

(b) We are proving that $\beta$ is injective. So suppose $a_j,b_j$ are elements of $\gamma_j(G)$, $k \leq j \leq n$, for which
\[
 a_k^{4^{n-k}} a_{k+1}^{4^{n-k-1}} \cdots a_{n} \equiv b_k^{4^{n-k}} b_{k+1}^{4^{n-k-1}} \cdots b_{n} \mod{ N_{n+1,k}(G) }.
\]
We prove that this implies $a_j \equiv b_j \mod{\gamma_j(G) \gamma_{j+1}(G)^4}$ for all $k \leq j \leq n$ by induction on $j$. By hypothesis, $a_l \equiv b_l \mod{\gamma_l(G) \gamma_{l+1}(G)^4}$ for all $k \leq l \leq j-1$, so $a_l^{4^{n-l-1}} \equiv b_l^{4^{n-l-1}} \mod{N_{n,k}(G)}$ by Lemma \ref{lemma-independent}. This now implies
\[
 a_j^{4^{n-j}} a_{j+1}^{4^{n-j-1}} \cdots a_{n} \equiv b_j^{4^{n-j}} b_{j+1}^{4^{n-j-1}} \cdots b_{n} \mod{ N_{n+1,k}(G) \cap \gamma_j(G)}.
\]
By Proposition \ref{prop-torsionFree1} (c), the intersection $N_{n+1, k}(G) \cap \gamma_j(G)$ equals $N_{n+1, j}(G)$. Note that the obtained congruence also holds when $j=k$. Hence 
\[
 a_j^{4^{n-j}} \equiv b_j^{4^{n-j}} \mod{ \gamma_j(G)^{4^{n-j}} \gamma_{j+1}(G) }.
\]
The group $\gamma_j(G)/\gamma_{j+1}(G)$ is abelian, so there exists an element $c \in \gamma_j(G)$ such that $a_j^{4^{n-j+1}} \equiv (b_jc^4)^{4^{n-j+1}} \mod{ \gamma_{j+1}(G) }$. The assumption that the factor group $\gamma_j(G)/\gamma_{j+1}(G)$ is also torison-free now gives us the desired congruence $a_j \equiv b_j \mod{ \gamma_j(G)^4 \gamma_{j+1}(G) }$.
\end{proof}

We will primarily be interested in the case when $G$ is a free group of finite rank. In this case, the stated theorem is fully applicable, and so the factor groups of the $\bar\lambda$-series are particularly easy to describe.

\begin{theorem} \label{thm-freeLambda}
 Let $F_r$ be the free group of rank $r$ and $n$ a positive integer.
 \begin{enumerate}[\indent \em (a)]
	\item The group $\bar\lambda_n(F_r)/\bar\lambda_{n+1}(F_r)$ can be embedded in $\bar\lambda_{n+1}(F_r)/\bar\lambda_{n+2}(F_r)$. Moreover, there exists a base $\mathcal A_n \mod{\bar\lambda_{n+1}}$ of $\bar\lambda_n/\bar\lambda_{i+n}$ such that $\mathcal A_n^4$ is independent in $\bar\lambda_{n+1}/\bar\lambda_{n+2}$.
	
  \item If $\mathcal{C}_n$ is the set of all basic commutators of weight $n$ in a fixed sequence, then the image of  $\mathcal{B}_n := \mathcal{C}_1^{4^{n-1}} \cup \mathcal{C}_2^{4^{n-2}} \cup \cdots \cup \mathcal{C}_n$ in $\bar\lambda_n(F_r)/\bar\lambda_{n+1}(F_r)$ is a base.

  \item The map $\varphi_n \colon \bar\lambda_{n}(F_r) \to \bar\lambda_{n+1}(F_r)/\bar\lambda_{n+2}(F_r)$ given by $\varphi_n(x) = x^4 \bar\lambda_{n+2}(F_r)$ is a homomorphism and $\ker \varphi_n = \bar\lambda_{n+1}$.
 \end{enumerate}
\end{theorem}
\begin{proof}
 For (a) and (b) we invoke Theorem \ref{thm-lambdaSeries} and refer the reader to \cite[Theorem 2.5]{bubboloni}, the proofs there are directly applicable in the case $p=2$. To prove (c), we first show that the maps $\varphi_n$ are indeed homomorphisms. When $n \geq 2$, this is true because for any $x,y \in \bar\lambda_{n}(F_r)$, we have $(xy)^4 = x^4 y^4 c$ for some $c \in \gamma_2(\bar\lambda_n(F_r)) \leq \bar\lambda_{2n}(F_r) \leq \bar\lambda_{n+2}(F_r)$. And when $n=1$, this is true by Proposition \ref{prop-torsionFree1} (a). We now show that $\ker \varphi_n(F_r) = \bar\lambda_{n+1}(F_r)$. Clearly $\bar\lambda_{n+1}(F_r) \leq \ker \varphi_n$, so pick any $x \in \lambda_n(F_r)$ with the property $x^4 \in \lambda_{n+2}(F_r)$. By (a), we can choose a base $\mathcal A_n \mod{ \bar\lambda_{n+1}(F_r) }$ of $\bar\lambda_n(F_r)/\bar\lambda_{n+1}(F_r)$ such that $\mathcal A_n^4 \mod{ \bar\lambda_{n+2}(F_r) }$ is independent in $\bar\lambda_{n+1}(F_r)/\bar\lambda_{n+2}(F_r)$. Expanding $x$ with respect to this base and using the fact that $\varphi_n$ is a homomorphism, we conclude $x \in \lambda_{n+1}(F_r)$.
\end{proof}
\section{Free $2$-central groups}
\label{free2Cent}

\noindent We now use the results of the previous section to extend the theory of \cite{bubboloni} to the case $p=2$. For any positive integers $n$ and $r \geq 2$, we set
\[
 {}^n G_r = \frac{F_r}{\bar\lambda_{n+1}(F_r)}.
\]

\begin{theorem}
The group ${}^n G_r$ is a $2$-central group of exponent $4^n$ and of nilpotency class $n$. Its $\bar \lambda$-series is equal to $\bar\lambda_i({}^n G_r)  = \bar\lambda_i(F_r)/\bar\lambda_{n+1}(F_r)$, and we have $\Omega_{2i}({}^n G_r) = \bar\lambda_{n+1-i}({}^n G_r)$. The order of ${}^n G_r$ equals $4^{b_1 + \cdots + b_n}$, where $b_i$ is the number of basic commutators of weight at most $i$ on $r$ generators. In particular, $|\Omega_2({}^n G_r)| = 4^{b_n}$.
\end{theorem}
\begin{proof}
The claim on the shape of the $\bar \lambda$-series is easily deduced by induction on $i$. Then the claims on the order, nilpotency class and exponent follow from the previous two theorems. The fact that ${}^n G_r$ is a $2$-central group follows from Theorem \ref{thm-freeLambda} (c), because the $\bar \lambda$-series of a group is a central series. The final claim on the $\Omega$-subgroups of ${}^n G_r$ is proved identically as in \cite[Theorem 3.5]{bubboloni}.
\end{proof}

Finally, we establish the importance of the group ${}^n G_r$ in the context of $p$-central groups.

\begin{theorem} \label{thm-homImage}
 Let $G$ be a $2$-central group with $r$ generators and of exponent $2^n$. Then $G$ is a homomorphic image of ${}^{\lceil \frac{n}{2} \rceil + 1} G_r$. Moreover, any finite $2$-group with $r$ generators is a homomorphic image of ${}^m G_r$ for some $m$.
\end{theorem}
\begin{proof}
 In any $2$-central group $G$ of exponent $4^n$, the $\Omega$-series $\cdots \leq \Omega_{2n-2} \leq \Omega_{2n} = G$ is central (see \cite{bubboloni0}), thus $\bar\lambda_i(G) \leq \Omega_{n-2i+2}(G)$ for all $i \geq 1$. In particular, $\bar\lambda_{\lceil \frac{n}{2} \rceil + 1}(G) = 1$. The presentation homomorphism $F_r \to G$ thus induces the required epimorphism. For the second part we only need to notice that for a finite $2$-group $G$ of exponent $2^n$ an of nilpotency class $c$, we have $\bar \lambda_{n+c-1}(G) = 1$ by Proposition \ref{prop-lambdaDescr}.
\end{proof}
\section{Schur multipliers}
\label{Mult}

\noindent In this section, we present two results concerning the Schur multipliers of $p$-central groups. The first one is a complete description of the Schur multiplier of the free $p$-central groups ${}^n G_r$ for any prime $p$.
 
\begin{theorem} \label{prop-MultiplierFree}
 The Schur Multiplier $M({}^n G_r)$ is elementary abelian (resp. free $\mathbb Z_4$-module for $p=2$) of rank $\sum_{i=2}^{n+1} c(r,i)$, where  $c(w,g)$ is the number of basic commutators of weight $w$ on $g$ generators (calculable by Witt's formula).
\end{theorem}
\begin{proof}
For odd $p$, Hopf's formula and \cite[VIII, 1.9]{huppertBlackburn} give us an isomorphism
\[
M({}^n G_r) \cong \frac{\gamma_2(F_r) \cap \lambda_{n+1}(F_r)}{[F_r, \lambda_{n+1}(F_r)]} = \frac{\gamma_2(F_r)^{p^{n-1}} \cdots \gamma_{n+1}(F_r)}{\gamma_2(F_r)^{p^{n}} \cdots \gamma_{n+2}(F_r)}.
\]
There is now a well-defined mapping
\[
\beta \colon \frac{\gamma_2(F_r)}{\gamma_2(F_r)^p \gamma_3(F_r)} \times \cdots \times \frac{\gamma_{n+1}(F_r)}{\gamma_{n+1}(F_r)^p \gamma_{n+2}(F_r)} \to  \frac{\gamma_2(F_r)^{p^{n-1}} \cdots \gamma_{n+1}(F_r)}{\gamma_2(F_r)^{p^{n}} \cdots \gamma_{n+2}(F_r)},
\]
given by $\beta(\bar{a}_2, \bar{a}_3, \dots, \bar{a}_{n+1}) = a_2^{p^{n-1}} a_3^{p^{n-2}} \cdots a_{n+1} \gamma_2(F_r)^{p^{n}} \cdots \gamma_{n+2}(F_r)$, where $a_i \in \gamma_i(F_r)$ and $\bar{a}_i = a_i \gamma_i(F_r)^p \gamma_{i+1}(F_r)$. Moreover, $\beta$ is a bijective homomorphism (cf. \cite[VIII, 1.8 and 1.9]{huppertBlackburn}).

For $p=2$, the proof is essentially the same, referring now to Theorem \ref{thm-lambdaSeries} and Proposition \ref{lemma-Nseries} (b) in providing the isomorphism between the product
\[
\frac{\gamma_2(F_r)}{\gamma_2(F_r)^4 \gamma_3(F_r)} \times \frac{\gamma_3(F_r)}{\gamma_3(F_r)^4 \gamma_4(F_r)} \times \cdots \times \frac{\gamma_{n+1}(F_r)}{\gamma_{n+1}(F_r)^4 \gamma_{n+2}(F_r)}
\]
and the multiplier $M({}^n G_r)$.
\end{proof}

Combining the above with Theorem \ref{thm-homImage} and \cite[Theorem 4.1]{bubboloni}, we obtain the following corollary, dualising a result of \cite{lubotzkyMannPowerful} on poweful $p$-groups.

\begin{corollary} \label{cor-ExpMdivG}
 Every finite $p$-group is a quotient of a $p$-group whose Schur multiplier is of exponent $p$ for odd $p$, $4$ for $p=2$.
\end{corollary}

Our second result is a positive solution to a special case of the conjecture stating that the exponent of the multiplier $M(G)$ divides the exponent of $G$ for any finite $p$-group $G$. We prove this by referring to Schur's theory of covering groups (cf. \cite[V]{huppert}).

\begin{proposition} \label{prop-expMDivG}
 Let $G$ be a finite $p$-central group. Then $\exp M(G) \leq \exp G$.
\end{proposition}
\begin{proof}
 Let $H$ be a covering group of $G$, so that $G \cong H/Z$ with $M(G) \cong Z \leq Z(H) \cap H'$, and let us denote $\exp G = p^e$. It suffices to prove that $H' \leq \Omega_e(H)$. Since $Z$ is central in $H$, we have $[\mho_{e}(H), H] = 1$. Using \cite[Theorem 2.5 (ii)]{fernandez}, we get
\[
 (H')^{p^e} \leq \prod _{i=1}^e [H^{p^{e-i}},{}_{i(p-1)+1}H].
\]
Since the group $G = H/Z$ is $p$-central, we have $(H/Z)^{p^{e-i}} \leq \Omega_i(H/Z) \leq Z_i(H/Z)$, from which we conclude $[H^{p^{e-i}}, {}_{i}H] \leq Z$. Since $i(p-1) \geq i$ for all $1 \leq i \leq e$ and $Z$ is central in $H$, the right hand side of the above product is trivial. This completes our proof.
\end{proof}
\section{Exponential rank}
\label{Exprank}

\noindent Following \cite{pmSchurMultPowerEnd}, we determine the exponential rank of $p$-central groups. Recall that the exponent semigroup of a finite $p$-group $G$ is 
\[
\mathcal{E}(G) = \{ n \in \mathbb{Z} \mid (xy)^n = x^n y^n \text{ for all } x,y \in G \}.
\]
Let $\exp (G/Z(G)) = p^e$. By \cite[Proposition 3.2]{pmSchurMultPowerEnd}, there exists a unique nonnegative integer $r$ such that $\mathcal{E}(G) = p^{e+r}\mathbb{Z} \cup (p^{e+r}\mathbb{Z} + 1)$. The exponential rank of the group $G$ is defined to be $\exprank (G) = r$. When $G$ is regular, its exponential rank is equal to $0$, and when $G$ is powerful, we similarly have $\exprank (G) = 0$ for odd $p$, $\exprank (G) = 1$ for $p=2$ if $G$ is not abelian. 

In the course of turning our attention to the exponential rank of $p$-central groups, let us first recall an old result of Laffey \cite[Theorem 3]{laffey2}.

\begin{theorem} \label{laff3}
 Let $G$ be a $p$-central group. Then $\exp G' = \exp G/Z(G)$.
\end{theorem}

Next we prove a lemma based on Hall's collection formula.

\begin{lemma} \label{lemma-productEss}
 Let $G$ be a $p$-central group. For any $x,y \in G$ such that $[x,y] \in \Omega_{i-\epsilon}(G)$ for some $i \geq 0$ and $\epsilon = 0$ for odd $p$, $\epsilon = 1$ for $p=2$, we have $x^{p^{i}} y^{p^{i}} = (xy)^{p^i}$.
\end{lemma}
\begin{proof}
 By Hall's collection formula, we have 
\[ 
 x^{p^{i}} y^{p^{i}} = (xy)^{p^i} c_2^{\binom{p^i}{2}}  c_3^{\binom{p^i}{3}} \cdots  c_{p^i}
\]
for some $c_j \in \gamma_j(\langle x,y \rangle)$. It thus suffices to prove that $\exp \gamma_j(\langle x,y \rangle)$ divides the binomial $\binom{p^i}{j}$ for all $2 \leq j \leq p^i$. Let first $p$ be odd. Since the $\Omega$-series of $G$ is central, we have $\gamma_j(\langle x,y \rangle) \leq \Omega_{i+2-j}(G)$ for all $j \geq 2$. Our claim is thus reduced to the fact that $p^{ i + 2 - j }$ divides $\binom{p^i}{j}$. Now let $[k]_p$ denote the largest integer such that $p^{[k]_p}$ divides $k$.  A straightforward induction shows $[\binom{p^i}{j}]_p = i - [j]_p$. The inequality in question is therefore equivalent to $[j]_p \leq j - 2$ for all $j \geq 2$, which evidently holds since $[j]_p \leq \lfloor \log_p j \rfloor$. When $p=2$, the series $1 = \Omega_0(G) \leq \Omega_2(G) \leq \cdots \leq G$ is central and thus  $\gamma_j(\langle x,y \rangle) \leq \Omega_{i+3-2j}(G)$ for all $j \geq 2$. It now analogously suffices to prove the inequality $i+3-2j \leq i - [j]_2$, which is again trivial.
\end{proof}

Finally, we state and prove our theorem.

\begin{theorem} \label{thm-exprank}
 Let $G$ be a $p$-central group.
\begin{enumerate}[\indent \em (a)]
 \item If $p$ is odd, then $\exprank(G) = 0$.
 \item If $p=2$ and $G$ is not abelian, then $\exprank(G) = 1$.
\end{enumerate}
\end{theorem}
\begin{proof}
 Suppose $\exp G/Z(G) = p^e$. By Theorem \ref{laff3}, we have $\gamma_2(G) \leq \Omega_e(G)$. For $p$ odd, Lemma \ref{lemma-productEss} shows that the mapping $x \mapsto x^{p^e}$ is an endomorphism of $G$, which is exactly our claim. Now let $p = 2$. The same reasoning shows that the mapping $x \mapsto x^{2^{e+1}}$ is an endomorphism of $G$, so we have $\exprank(G) \leq 1$. If $G$ is abelian, then clearly $\exprank(G) = 0$. Suppose now that there exists a nonabelian $2$-central group $G$ with $\exprank(G) = 0$. Since $\gamma_2(G) \leq \Omega_e(G)$, we obtain $\gamma_i(G) \leq \Omega_{e+2-i}(G)$ for all $2 \leq i \leq e$, and consequently $\gamma_{2^i}(G) \leq \Omega_{e-i}(G)$. Lemma \ref{lemma-hallPet} thus implies
 \[
 (xy)^{2^e} \equiv x^{2^e} y^{2^e} \mod{\gamma_2(\langle x,y \rangle)^{2^{e-1}}}.
 \]
The corresponding terms in $\gamma_2(\langle x,y \rangle)^{2^{e-1}}$ can be computed using the commutator collection process \cite[pp. 165--166]{hall}. We obtain 
\[
 (xy)^{p^e} = x^{p^e} y^{p^e} [y,x]^{\binom{2^e}{2}} [y,x,x]^{\binom{2^e}{3}} [y,x,y]^{\binom{2^e}{2} + 2{\binom{2^e}{3}}}.
\]
Since $\exp \gamma_2(G) = 2^e$, the fact that the $\Omega$-series is central in $G$ gives $\exp \gamma_3(G) \leq 2^{e-1}$. We are also assuming $\exprank(G) = 0$, so the above equality implies $[y,x] \in \Omega_{e-1}(G)$ for all $x,y \in G$, but this is clearly a contradiction.
\end{proof}

Note that the obtained result coincides with the one for powerful $p$-groups.
\vspace{0.5\baselineskip}

\noindent {\bf Acknowledgement.} ~ The author wishes to express his gratitude to Primož Moravec for adding his valuable suggestions and guidance.


\end{document}